\documentclass[11pt]{article}
\usepackage{amsfonts,amsmath,tikz}
\usepackage{epsfig}
\usepackage{latexsym}
\usepackage{color}
\usepackage{epsf,amssymb}

\usepackage{amsthm}
\usepackage{yfonts}
\usepackage{hyperref}
\usepackage{ifthen}
\usepackage{enumitem,enumerate}
\usepackage[T1]{fontenc} 
\usepackage{diagbox}
\usepackage{array}
\usetikzlibrary{arrows.meta}
\usepackage[normalem]{ulem}

\usepackage{algorithm}
\usepackage{algpseudocode}

\usepackage{float} 
\floatstyle{plaintop}
\restylefloat{table} 

\newtheorem{theorem}{Theorem}
\newtheorem{lemma}[theorem]{Lemma}

\newtheorem{corollary}[theorem]{Corollary}
\newtheorem{proposition}[theorem]{Proposition}
\newtheorem{problem}[theorem]{Problem}

\def\vertex(#1){\put(#1){\circle*{2}}}
\def\vertexo(#1){\put(#1){\circle{2}}}
\def\vert(#1){\put(#1){\circle*{1.5}}}
\def\verto(#1){\put(#1){\circle{1.5}}}
\def\lab(#1)#2{\put(#1){\makebox(0,0)[c]{#2}}}
\newcommand{\g}{\gamma}
\newcommand{\io}{\iota}
\newcommand{\es}{\alpha_{\rm es}}
\newcommand{\al}{\alpha}

\newcommand{\EPN}{\mathrm{epn}}

\tikzset{My Style/.style={draw, circle, fill=black,scale=0.3}} 

\title{On externally supported independence number of graphs}

\author{
Dragana Bo\v zovi\'{c}\textsuperscript{a,b} \and
Iztok Peterin \textsuperscript{a.c} \and
Adriana Roux \textsuperscript{d} \and 
Aleksandra Tepeh\textsuperscript{a,b,c}
}     

\begin{document}
\maketitle

\begin{center}
\textsuperscript{a}University of Maribor, FERI, Koroška cesta 46, 2000 Maribor, Slovenia \\
\textsuperscript{b}University of Maribor, FNM, Koroška cesta 160, 2000 Maribor, Slovenia \\
\textsuperscript{c}IMFM, Jadranska ulica 19, 1000 Ljubljana, Slovenia \\
\textsuperscript{d}Department of Mathematical Sciences, Stellenbosch University, Stellenbosch, 7600, South Africa
\end{center}

\begin{abstract}
We introduce the \emph{externally supported independence number} $\es(G)$ of a graph $G$ as the maximum cardinality of an independent set $B$ with an additional condition, that vertices from $N(B)$ are dominated by vertices in $V(G)-N[B]$. This parameter yields an improved upper bound on the isolation number $\io(G)$. We show that computing $\es(G)$ is NP-hard, while for trees we present a linear-time algorithm. We also establish several sharp bounds on  $\es(G)$ for general graphs, with additional refined results for trees. In several cases, we completely describe the extreme graph classes attaining these bounds.   
\end{abstract}

\noindent
{\bf Keywords:} externally supported independence number, isolation number, independence number, tree, computational complexity

\medskip

\noindent
{\bf Mathematics Subject Classification:} 05C69

\section{Introduction} 

In domination we are looking for the smallest set of vertices such that every vertex not in the set is adjacent to a vertex in the set, we say that such a set dominates the graph. Domination is often use for modeling facility location problems. For example, let a city be represented by a graph; determining where to place service facilities, so that people from every part of the city have the service in question close to their home is a facility location problem that can be solved with domination in graph theory. An optimal solution might not always be possible to implement, for example, desired locations may already be occupied by other activities, or too many services of the same kind may be too expensive, to mention just two. It is therefore often necessary to consider suitable relaxations.

One recent relaxation closely related to domination is isolation, where the chosen set is not required to dominate all the vertices of the graph. With isolation the focus falls on avoiding certain structures among the undominated vertices. Let $\mathcal{F}$ denote the family of graphs to avoid. A set of vertices will be an $\mathcal{F}$-isolating set if none of the graphs in $\mathcal{F}$ occurs as an induced subgraph of the graph that remains after removing the set and its neighbors. The minimum number $\io_{\mathcal{F}}(G)$ of such vertices is called the $\mathcal{F}$-isolation number of a graph $G$.

In 2010 Lewis et al.~\cite{LHHF} introduced the concept of isolation for $\mathcal{F}=\{K_2\}$ under the name of \emph{vertex-edge domination}. They characterized trees satisfying $\io_{\mathcal{F}}(T)=\g(T)$. \.{Z}yli\'{n}ski \cite{Zyli} and Ziemann proved a general upper bound on vertex-edge domination number of a graph and then later \.{Z}yli\'{n}ski \cite{ZiZy} improved this result for cubic graphs. Vertex-edge domination is a special case of $\mathcal{F}$-isolation where $\mathcal{F}=\{K_2\}$; in this case we refer simply to the \emph{isolation number} $\io(G)$ of a graph. Caro and Hansberg \cite{caro} presented several lower bounds together with the general upper bound $\io(G)\leq\frac{n(G)}{3}$ for connected graphs. Graphs attaining this upper bound were characterized among unicyclic and block graphs by Lema\'nska et al.~\cite{LMSS}, and a structural description was later given by Boyer and Goddard \cite{BoGo}. Recently the isolation number of Cartesian products and generalized Sierpi\'nski graphs was considered by Bre\v sar et al. \cite{BDJK}. Other choices of $\mathcal{F}$ have been studied as well, including all cycles \cite{Borg}, $C_4$ \cite{BaBS} and $K_k$ \cite{ChCZ}. Variants such as total isolation \cite{BoGH} and independent isolation \cite{BoGo2} have also been introduced.

To determine the isolation number $\io(G)$ of a graph $G$, one seeks a set $A\subseteq V(G)$ of minimum cardinality such that  $V(G)$ is partitioned into $A$, the neighborhood of $A$, and a third set that induces an independent set in $G$. In other words, all vertices that are not dominated by $A$ form an independent set in $G$. While many publications on the isolation number follow domination-oriented approach, we turn the table and put independent sets into the spotlight.
More precisely,  we study independent sets $B$ with the property that every vertex in $N(B)$
has a neighbor outside the closed neighborhood $N[B]$; we call such sets \emph{externally
supported independence sets}. In other words, we are interested in the independent set not dominated by an isolating set.

The paper is organized as follows. In the next section we present the notation and necessary definitions. In Section \ref{com} we introduce externally supported independence sets and the corresponding
parameter, and we study the computational complexity of the associated decision problem. In Section \ref{num} we derive several bounds for this new parameter and compare it with the
independence number. We also present an upper bound in terms of the order and the minimum
and maximum degree, and we characterize all graphs that admit no nonempty externally
supported independence set.
In Section \ref{tr} we focus on trees and characterize extremal trees for a lower and an upper bound, and present a linear-time algorithm for externally supported independence number of a tree. The concluding section presents possible  connections to other graph parameters and some ideas for further research.


\section{Preliminaries}

We use standard notation $[k]$ for the set $\{1,2,\dots,k\}$, and $K_n,\overline{K}_n,C_n,P_n$ for the complete graph, the edgeless graph, the cycle and the path, respectively, on $n$ vertices.  Throughout the paper, graphs are finite and simple.

Let $G$ be a graph on $n(G)$ vertices.
For $v\in V(G)$ we write $N_G(v)$ for its open neighborhood and $N_G[v]=N_G(v)\cup \{v\}$ for its closed neighborhood. 
For $A\subseteq V(G)$ we set $N_G[A]=\cup_{a\in A}N_G[a]$, $N_G(A)=N_G[A]-A$. When no confusion is likely, we omit the subscript $G$ and write $N(v)$, $N[v]$, $N(A)$,
and $N[A]$.
By $G[S]$ we denote the \emph{induced
subgraph} of $G$ on $S\subseteq V(G)$, that is, the graph with vertex set $S$ and edge set $E(G[S])=\{uv\in E(G): u,v\in S\}$.

A set $S\subseteq V(G)$ is an \emph{independent} set of $G$ if there are no edges in $G$ between vertices of $S$. Let $A\subseteq V(G)$ and $B=V(G)-N_G[A]$. The set $A$ is a \emph{dominating set} if $B=\emptyset$, and it is an \emph{isolating set} if $B$ is an independent set. We say that a vertex $v$ \emph{dominates} all vertices in $N[v]$. Clearly, every dominating set is isolating, but not every isolating set is dominating. 

The maximum cardinality of an independent set of $G$ is the \emph{independence number} $\al(G)$, the minimum cardinality of a dominating set of $G$ is the \emph{domination number} $\g(G)$ and the minimum cardinality of an isolating set of $G$ is the \emph{isolation number} $\io(G)$. An independent set of cardinality $\al(G)$ is called an $\al
(G)$-set, a dominating set of cardinality $\g(G)$ is a $\g(G)$-set, and an isolating set of cardinality $\io(G)$ is called $\io(G)$-set.

For $v\in V(G)$ let ${\rm deg}(v)=|N(v)|$ be the \emph{degree} of $v$. A vertex $v$ is \emph{universal} if ${\rm deg}(v)=n(G)-1$,  a \emph{leaf} if ${\rm deg}(v)=1$, and an \emph{isolated vertex} if ${\rm deg}(v)=0$. Let $L(G)$ be the set of all leaves of $G$ and set $\ell(G)=|L(G)|$. Locally, for each $v\in V(G)$ let $L(v)$ be the set of leaves adjacent to $v\in V(G)$ and set $\ell(v)=|L(v)|$. The neighbor of a leaf $v$ is called its \emph{support vertex}. Let $S(G)$ denote the set of all support vertices of $G$ and set $s(G)=|S(G)|$. For $v\in V(G)$ let  $S(v)$ be the set of all support vertices adjacent to $v$ and set $s(v)=|S(v)|$.

Distinct vertices $u$ and $v$ are \emph{twins} if $N[u]=N[v]$. Let $u$ and $v$ be twins. A vertex $w$ is a \emph{support of a twin} $u$ (and $v$) if $N[u]\subset N[w]$. We denote by $TS(G)$ the set of all twin support vertices of $G$.


\section{Definition and complexity}
\label{com}

Let $A$ be an isolation set of $G$. In this paper we focus on $B=V(G)-N[A]$, the independent set left after removing $A$ together with
its closed neighborhood. This viewpoint connects isolation to the well-developed theory of independent sets and motivates studying how large such a residual independent set can be. This leads to a refined notion of independent sets tailored to isolation.

First we recall the following bound from~\cite{caro}:
\begin{equation}\label{low_bound}
\io(G)\geq \left\lceil\frac{n(G)+1-\al(G)}{\Delta(G)+1}\right\rceil.
\end{equation}
Indeed, let $A$ be an $\io(G)$-set. Then $B=V(G)-N[A]$ is independent and,
since $A\neq\emptyset$, for any $x\in A$ the set $B\cup\{x\}$ is independent. Hence $|B|\le \al(G)-1$. Moreover, each vertex of $A$
dominates at most $\Delta(G)+1$ vertices, so $(\Delta(G)+1)|A|\ge n(G)-|B|$. Therefore
$$(\Delta(G)+1)\io(G)\ge n(G)-(\al(G)-1)=n(G)+1-\al(G),$$
which yields~\eqref{low_bound}. \\

Let $A$ be an isolating set of a graph $G$ and consider the remainder $B=V(G)-N[A]$.
By definition, $B$ is an independent set. However, not every independent set can occur
in this way: the condition that $B$ arises after removing $A$ together with its closed
neighborhood imposes an additional local constraint. That is, every vertex $v\in N(B)$
must be dominated by $A$, and since $A\subseteq V(G)-N[B]$, this is possible
only if $v$ has a neighbor outside $N[B]$. This observation leads to the following notion.

\medskip

A set $B \subseteq V(G)$ is called an \emph{externally supported independence set} (an \textit{ESI--set} for short) if $B$ is an independent set and every $v\in N(B)$ has a neighbor in $D=V(G)-N[B]$. The \emph{externally supported independence number} of $G$ (shortly the \emph{ESI--number}), denoted by $\es(G)$, is the maximum cardinality of an ESI-set in $G$. An ESI-set of cardinality $\es(G)$ will be referred to as $\es(G)$-set. For convenience, we shall refer to the requirement that every vertex in $N(B)$ has a neighbor in $D$ as the \emph{ES-condition}. 

This new definition results in the following reformulation of (\ref{low_bound}).

\begin{theorem} 
\label{gen_bound}
For any graph $G$ we have 
$$\io(G)\geq \left\lceil\frac{n(G)-\es(G)}{\Delta(G)+1}\right\rceil.$$ 
\end{theorem}

\begin{proof}
Let $A$ be an $\io(G)$-set with $A'=V(G)-N[A]$.
Every vertex from $A$ dominates at most $\Delta(G)+1$ vertices of $G$, and on the other hand, $A$ needs to dominate $n(G)-|A'|$ vertices. 
Clearly, $A'$ is an ESI-set, since it is independent and each $v\in N(A')$ has a neighbor in
$A\subseteq V(G)-N[A']$. Thus $|A'|\le \es(G)$, and
we derive
$$(\Delta(G)+1)\io(G)\geq n(G)-|A'|\geq n(G)-\es(G),$$
which gives us the bound in the theorem. 
\end{proof}

Rearranging the inequality in Theorem \ref{gen_bound} yields the following lower bound on $\es(G)$.

\begin{corollary} 
\label{alpha-bound}
For any graph $G$ we have 
$$\es(G)\geq n(G)-\io(G)(\Delta(G)+1).$$
\end{corollary}

To make the bound in Theorem~\ref{gen_bound} effective, one needs to be able to compute or
estimate $\es(G)$. However, as expected, the following problem turns out to be NP-complete.

\begin{center}
\fbox{\parbox{0.8\linewidth}{\noindent
\textbf{\textsc{ESI (Externally Supported Independent) Set problem}}\\[.8ex]
\begin{tabular*}{.93\textwidth}{rl}
{\em Instance:} & A graph $G$ and a positive integer $k$.\\
{\em Question:} & Does $G$ contain an ESI-set of size at least $k$?
\end{tabular*}
}}
\end{center}

To see this, we describe the following construction of a graph $G$ from a given graph $H$. Let $H$ be any graph and let $\ell>n(H)$. We construct a new graph $M(H,K_{\ell})$ as follows: take one copy of $H$, one copy of the complete graph $K_{\ell}$, and add a perfect matching between the vertices of $H$ and an arbitrary set of $n(H)$ vertices of $K_{\ell}$. Clearly, the graph $M(H,K_{\ell})$ can be constructed in polynomial time. An example of $M(C_5,K_6)$ is shown in Figure \ref{grafMprimer}. We begin with the following lemma.

\begin{figure}[!ht]
\begin{center}
	\begin{tikzpicture}[scale=0.5]
		\tikzstyle{crn}=[circle,fill=black,draw, inner sep=0pt, minimum size=5pt]
			
		\begin{scope}
			\node (a)[crn] at (0 cm,0 cm){};
            \node (b)[crn] at (4 cm,0 cm){};
            \node (c)[crn] at (6 cm,1.5 cm){};
            \node (d)[crn] at (2 cm,2.5 cm){};
            \node (e)[crn] at (-2 cm,1.5 cm){};
            
            \draw (a)--(b)--(c)--(d)--(e)--(a);
		\end{scope}
			
		\begin{scope}[shift={(0,3)}]
			\node (a1)[crn] at (0 cm,0 cm){};
            \node (b1)[crn] at (4 cm,0 cm){};
            \node (c1)[crn] at (6 cm,1.5 cm){};
            \node (d1)[crn] at (2 cm,2.5 cm){};
            \node (e1)[crn] at (-2 cm,1.5 cm){};
            
            \draw (a1)--(b1)--(c1)--(d1)--(e1)--(a1);
            \draw (a1)--(c1)--(e1)--(b1)--(d1)--(a1);
		\end{scope}

		\begin{scope}[shift={(0,8)}]
			\node (t)[crn] at (2 cm,0 cm){};
            \draw (t)--(a1);
            \draw (t)--(b1);
            \draw (t)--(c1);
            \draw (t)--(d1);
            \draw (t)--(e1);
		\end{scope}

        \draw (a)--(a1);
        \draw (b)--(b1);
        \draw (c)--(c1);
        \draw (d)--(d1);
        \draw (e)--(e1);
    \end{tikzpicture}
\caption{The graph $M(C_5,K_6)$.}
\label{grafMprimer}
\end{center}
\end{figure}
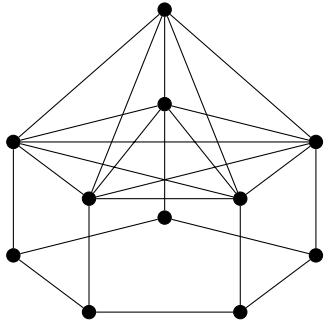

\begin{lemma} 
\label{lemma:M}
Let $H\ncong K_n$ be a graph and let $\ell>n(H)$ be an integer. If $G=M(H,K_{\ell})$, then a set $B$ is an $\es(G)$-set if and only if it is an $\al(H)$-set.   
\end{lemma}

\begin{proof}
Since $H\ncong K_n$, we have $\al(H)\geq 2$. 
Let $V(H)=\{v_1,\dots,v_{n(H)}\}$, $V(K_{\ell})=\{u_1,\dots,u_{\ell}\}$, and for each $i\in[n(H)]$, let $u_i$ be the vertex of $K_{\ell}$ matched to $v_i$ in the construction $M(H,K_{\ell})$.

We first show that every independent set of $H$ is an ESI-set of $G$. Let $A\subseteq V(H)$ be independent. If $v_i\in N(A)$, then its matched neighbor $u_i \notin N[A]$. Similarly, if $u_j\in N(A)$, then $v_j\in A$, and $u_j$ has a neighbor $u_{\ell}\notin N[A]$, since $\ell>n(H)$ and $u_{\ell}$ has no neighbors in $H$.
Therefore, $A$ satisfies the ES-condition. In particular, every $\al(H)$-set is an ESI-set of $G$, and hence $\es(G)\ge \al(H) \ge 2$.

Next we show that no $\es(G)$-set contains a vertex of $K_{\ell}$. Suppose for a contradiction that an $\es(G)$-set $B$ contains some $u_j\in V(K_{\ell})$. Since $u_j\in B$ and $K_{\ell}$ is complete, it follows that $V(K_{\ell})\subseteq N_G[B]$.
Now let $i\in[n(H)]$ with $i\neq j$. If $v_i\in B$, then $u_i\in N(B)$, but all neighbors of $u_i$ lie in $N_G[B]$, so $u_i$ has no neighbor in $V(G)-N_G[B]$, a contradiction with the ES-condition. Hence $v_i\notin B$ for every $i\neq j$. Also, if $j\leq n(H)$, then $v_j\notin B$, since $u_j$ is adjacent to $v_j$ and $B$ is independent. Therefore $B=\{u_j\}$, contradicting $\es(G)\ge 2$. Thus no $\es(G)$-set contains a vertex of $K_{\ell}$, and every $\es(G)$-set is contained in $V(H)$.

It follows that every $\es(G)$-set is an independent subset of $V(H)$, and hence has size at most $\al(H)$. Since every $\al(H)$-set is an ESI-set of $G$, we already know that $\es(G)\ge \al(H)$. Therefore $\es(G)=\al(H)$, which completes the proof.
\end{proof}

We have a polynomial-time construction $M(H,K_{\ell})$ for every non-complete graph $H$ and $\ell>n(H)$ such that $\es(G)=\al(H)$. 
The \textsc{ESI Set} problem belongs to NP, because for a given set $B$ we can verify in polynomial time that $B$ is independent and that every vertex in $N(B)$ has a neighbor in $V(G)\setminus N[B]$.
Since the \textsc{Independent Set} problem is NP-complete \cite{Garey}, this yields the following theorem.

\begin{theorem} 
\label{complex}
The \textsc{ESI Set} problem is NP-complete.
\end{theorem}


\section{Some exact results and bounds for ESI-number}
\label{num}

We begin with some basic bounds on $\es(G)$. Trivially, we have $\es(G)\geq 0$. Since every ESI-set is independent, 
$\es(G)$ is immediately comparable to the independence number $\alpha(G)$, and the trivial upper bound
$\es(G)\le \alpha(G)$ holds. Note that equality occurs for edgeless graphs $G=\overline{K}_n$ where the
whole vertex set is independent and $\es(G)=\alpha(G)=n$ follows. Hence, throughout the remainder of this section we assume that $G$ has at least one edge.

We continue with exact results for some well-known graph families, starting with the join $G\vee H$ of two graphs $G$ and $H$. Recall that $G\vee H$ contains disjoint copies of $G$ and $H$ together with all edges between $V(G)$ and $V(H)$. 


\begin{proposition}
\label{join}
If $G$ and $H$ are graphs of order $p\geq 1$ and $r\geq 1$, respectively, then 
\begin{equation*}
\es(G\vee H)=\left\{\begin{array}{ll}
\max\{\es(G)-1,\es(H)-1\} & \mbox{if } G\cong \overline{K}_p\ \&\ H\cong \overline{K}_r, \\[0.15cm]
\max\{\es(G),\es(H)-1\} & \mbox{if } G\ncong \overline{K}_p\ \&\ H\cong \overline{K}_r, \\[0.15cm]
\max\{\es(G)-1,\es(H)\} & \mbox{if } G\cong \overline{K}_p\ \&\  H\ncong \overline{K}_r, \\[0.15cm]
\max\{\es(G),\es(H)\} & \mbox{if }  G\ncong \overline{K}_p\ \&\  H\ncong \overline{K}_r. \\[0.15cm]
\end{array}\right.
\end{equation*}
\end{proposition}

\begin{proof}
Let $B_G$ and $B_H$ be an $\es(G)$-set and an $\es(H)$-set, respectively. Whenever these
sets are nonempty, choose  $g\in B_G$ and $h\in B_H$. Since $G\vee H$ contains all edges between $V(G)$ and $V(H)$, every independent
set of $G\vee H$ is contained entirely in $V(G)$ or entirely in $V(H)$. Thus, it is straightforward to verify that the following sets are ESI-sets of $G\vee H$:
\begin{itemize}
\item $B_G-\{g\}$ and $B_H-\{h\}$ when $G\cong \overline{K}_p$ and $H\cong \overline{K}_r$;
\item $B_G$ and $B_H-\{h\}$ when $G\ncong \overline{K}_p$ and $H\cong \overline{K}_r$;
\item $B_G-\{g\}$ and $B_H$ when $G\cong \overline{K}_p$ and $H\ncong \overline{K}_r$;
\item $B_G$ and $B_H$ when $G\ncong \overline{K}_p$ and $H\ncong \overline{K}_r$.
\end{itemize}
Hence, $\es(G\vee H)\geq R$, where $R$ is the right-hand side of the claimed equality.

Assume for a contradiction that $\es(G\vee H)>R\geq 0$. Let $B\neq\emptyset$ be an $\es(G\vee H)$-set. As observed before, we have either $B\subseteq V(G)$ or $B\subseteq V(H)$. Without loss of generality, assume that
$B\subseteq V(H)$. If $|B|>\es(H)$, then the ES-condition fails for some vertex in $H$ and for the same vertex the ES-condition fails also in $G\vee H$, a contradiction. So, $|B|\leq \es(H)$, and the only remaining possibility is $|B|=\es(H)$ and $R=\es(H)-1$, which implies $H\cong \overline{K}_r$. In this case $B=V(H)$, and every vertex from $G$ violates the ES-condition, a contradiction again. So, $\es(G\vee H)\leq R$ and the desired equality follows.
\end{proof}

In the proof of the next result we will use the fact that several well-known graph families can be expressed as joins. For example, complete graphs satisfy $K_n=K_p\vee K_{n-p}$ for $1\leq p\leq n-1$, complete bipartite graphs  satisfy $K_{p,q}=\overline{K}_p\vee \overline{K}_q$ for $p,q\geq 1$, for wheels we have $W_n=K_1\vee C_{n-1}$ for $n\geq 4$, and for fans $F_n=K_1\vee P_{n-1}$ for $n\geq 3$.

\begin{proposition}
\label{families}
The following holds.
\begin{itemize}
\item[(i)] For every $n\geq 3$ we have $\es(P_n) = \lceil\frac n4\rceil$.
\item[(ii)] For every $n\geq 4$ we have $\es(C_n) = \lfloor\frac n4\rfloor$. 
\item[(iii)] For every $n\geq 2$ we have $\es(K_n) = 0$.
\item[(iv)] For every $n\geq 3$ we have $\es(K_{1,n-1}) = n-2$.
\item[(v)] For every $m,n\geq 2$ we have $\es(K_{m,n}) = \max\{m-1,n-1 \}$.
\item[(vi)] For every $n\geq 4$ we have $\es(W_n) = \lfloor\frac {n-1}{4}\rfloor$.
\item[(vii)] For every $n\geq 4$ we have $\es(F_n) = \lceil\frac {n-1}{4}\rceil$.
\end{itemize}
\end{proposition}

\begin{proof}
\textit{(i)} Let $B\subseteq V(P_n)$ be an $\es(P_n)$-set where $P_n=x_1x_2\ldots x_n$. Clearly, $\es(P_3)=\es(P_4)=1$. Now assume that $n\geq 5$. We claim that the distance in $P_n$ between any two distinct vertices of $B$ is at least 4. Indeed, since $B$ is independent, distance 1 is impossible. If two vertices $x_i, x_{i+2}\in B$, then $x_{i+1}\in N(B)$ has no neighbor in $V(P_n)-N[B]$, contradicting the ES-condition. Similarly, if $x_i, x_{i+3}\in B$, then $x_{i+1},x_{i+2}\in N(B)$ have no neighbor in $V(P_n)-N[B]$, again contradicting the ES-condition. Hence any two vertices of $B$ are at distance at least 4 and $\es(P_n)\le \lceil \frac n4 \rceil$. To show the opposite inequality let $k=\lfloor \frac{n-1}{4} \rfloor$, consider the set $B=\{x_{4i-3}\in V(P_n); \; i\in [k]\}\cup\{x_n\}$. It is straightforward to verify that $B$ is independent and every vertex of $N(B)$ has a neighbor outside $N[B]$, hence $B$ satisfies the ES-condition. Therefore, $\es(P_n)\geq|B|=\lceil \frac n4 \rceil$.

\medskip 
\textit{(ii)} Let $C_n=x_1x_2\ldots x_n x_1$. The argument for the upper bound is analogous to the one for paths. In particular, if $B\subseteq V(C_n)$ is an $\es(C_n)$-set and $x_i\in B$, then none of $x_{i-3},x_{i-2},x_{i-1},x_{i+1},x_{i+2},x_{i+3}$ (indices taken modulo $n$) can belong to $B$, since otherwise one obtains a vertex in $N(B)$ whose both neighbors lie in $N[B]$, contradicting the ES-condition. Hence, vertices of $B$ are pairwise at distance at least 4, which implies $4|B|\leq n$ and therefore $\es(C_n)=|B|\leq \lfloor \frac n4 \rfloor$. To show the opposite inequality, let $k=\lfloor \frac n4 \rfloor$ and consider the set $B=\{x_{4i-3}\in V(C_n); \; i\in [k]\}$. It is straightforward to verify that $B$ is independent and every vertex of $N(B)$ has a neighbor outside $N[B]$, hence $B$ satisfies the ES-condition. Therefore, $\es(C_n)\geq|B|=\lfloor \frac n4 \rfloor$.

\medskip
For \textit{(iii)} observe that $N[v]=V(K_n)$ holds for every vertex $v\in K_n$ and no vertex can be in an ESI-set. So, ESI-set of $K_n$ is empty and thus $\es(K_n)=0$. In addition, \textit{(iv), (v), (vi), (vii)} follow directly from Proposition \ref{join}, and additionally by \textit{(ii)} for $W_n$, and \textit{(i)} for $F_n$.
\end{proof}

We can use \textit{(i)} and \textit{(ii)} of Proposition \ref{families} to motivate the study of $\es(G)$. Let $G$ be $P_n$ or $C_n$. Then $\al(G)$ is about $\frac{n}{2}$, whereas $\es(G)$ is about $\frac{n}{4}$. Since $\Delta(G)=2$, 
the bound (\ref{low_bound}) gives a lower bound 
on $\io(G)$ of about $\frac{n}{6}$, whereas Theorem \ref{gen_bound} gives a lower bound of about $\frac{n}{4}$, which is substantially stronger.

\medskip

If a graph $G$ has at least one edge we can improve the trivial bound as follows.

\begin{theorem} 
\label{alpha_up_bound}
For any graph $G$ with at least one edge we have 
$$0\leq\es(G)\leq\al(G)-1.$$
\end{theorem}

\begin{proof}
The inequality $\es(G)\ge 0$ is trivial. To prove the upper bound, let $uv\in E(G)$ and suppose for a contradiction that $\es(G)=\al(G)$. Let $B$ be an $\es(G)$-set. If $u\in B$, then $v\in N(B)$, and by the ES-condition $v$ has a neighbor $x$ in $V(G)-N[B]$. Since $x\notin N[B]$, the vertex $x$ is not adjacent to any
vertex of $B$, and hence $B\cup\{x\}$ is an independent set of $G$ with more vertices than $B$, a contradiction. The same argument applies if $v\in B$. 
Thus, $u,v\notin B$. If $\{u,v\}\cap N(B)\neq\emptyset$, then the ES-condition yields $x\in V(G)-N[B]$
and we argue exactly as in the previous case to get a contradiction. Therefore $u,v\notin N[B]$. Therefore, $u,v\notin N[B]$. In particular, $v$ is not adjacent to any vertex of $B$, and since also $v\notin B$, the set
$B\cup\{v\}$ is independent of size $|B|+1$, a final contradiction. Hence $\es(G)\le \al(G)-1$.
\end{proof}

We next characterize graphs with $\es(G)=0$, that is, graphs that admit no nonempty ESI-set. Our main tool is the following local criterion, which identifies vertices that cannot belong to any ESI-set.

\begin{lemma}
\label{notinESI}
Let $G$ be a graph and let $x\in V(G)$. Then  $x$ does not belong to any ESI-set of $G$ if and only if there exists $y\in N(x)$ such that $N[y]\subseteq N[x]$.
\end{lemma}

\begin{proof} 
Assume first that $x$ belongs to no ESI-set. In particular, the singleton $\{x\}$ is not an ESI-set. Since $\{x\}$ is independent, it must fail  the ES-condition, and hence there exists $y\in N(x)$ that has no neighbor in $V(G)-N[x]$. Equivalently, $N(y)\subseteq N[x]$, and since $y\in N(x)$, we obtain $N[y]\subseteq N[x]$.

Conversely, assume that there exists $y\in N(x)$ with $N[y]\subseteq N[x]$.
Let $B$ be any independent set with $x\in B$. Then $y\in N(B)$ and $N(y)\subseteq N[y]\subseteq N[x]\subseteq N[B]$, so $y$ has no neighbor in $V(G)-N[B]$. Hence $B$ fails the ES-condition and therefore
cannot be an ESI-set. Since $B$ was arbitrary, $x$ belongs to no ESI-set.
\end{proof}

As an immediate consequence of Lemma~\ref{notinESI}, we obtain the following forbidden-vertex
conditions for ESI-sets.

\begin{corollary}
\label{notESI1}
If $B$ is an ESI-set of a graph $G$, then $B$ contains neither a universal vertex, nor a
support vertex, nor a twin, nor a support of a twin.
\end{corollary}

We next present two characterizations of graphs with $\es(G)=0$, that is, graphs admitting no ESI-set.
First we give a neighborhood-inclusion condition, and then we refine it into an explicit description of the
vertex types that force $\es(G)=0$.

\begin{theorem}
\label{zero}
The following statements are equivalent for a graph $G$:
\begin{itemize}
\item[(i)] $\es(G)=0$;
\item[(ii)] every vertex $x\in V(G)$ has a neighbor $y\in N_G(x)$ such that $N_G[y]\subseteq N_G[x]$;
\item[(iii)] every vertex of $G$ is a universal vertex or a twin or support of a twin. 
\end{itemize}
\end{theorem}

\begin{proof} 
$(i)\Rightarrow (ii)$. Assume that $\es(G)=0$. Since no nonempty ESI-set exists, no vertex $x\in V(G)$ can belong to an ESI-set. By Lemma \ref{notinESI}, this implies that for every $x\in V(G)$ there exists a neighbor $y\in N(x)$ such that $N[y]\subseteq N[x]$.

$(ii)\Rightarrow (i)$. Assume that for every $x\in V(G)$ there exists a neighbor $y\in N(x)$ such that $N[y]\subseteq N[x]$. By  Lemma \ref{notinESI} $x$ belongs to no ESI-set and therefore no nonempty ESI-set exists. So, $\es(G)=0$ follows.

$(i)\Rightarrow(iii)$.
Assume that there exists a vertex $v_1\in V(G)$ that is neither
universal, nor a twin, nor a support of a twin. We show that then $\es(G)>0$.

If $v_1$ belongs to some ESI-set, then $\es(G)>0$ and we are done. Otherwise, $v_1$ belongs to
no ESI-set. By Lemma~\ref{notinESI}, there exists a neighbor $u\in N(v_1)$ such that
$N[u]\subseteq N[v_1]$. Choose such a vertex $v_2\in N(v_1)$ with $N[v_2]\subseteq N[v_1]$
that is minimal with respect to set inclusion among all vertices $u\in N(v_1)$ satisfying
$N[u]\subseteq N[v_1]$.

If $v_2$ belongs to some ESI-set, then $\es(G)>0$ and we are done. Otherwise, $v_2$ belongs to
no ESI-set, and Lemma~\ref{notinESI} yields a neighbor $v_3\in N(v_2)$ such that
$N[v_3]\subseteq N[v_2]$. By the choice of $v_2$, this inclusion cannot be strict,
and hence $N[v_3]=N[v_2]$. Thus $v_2$ and $v_3$ are twins. Since $v_1$ is not a twin, we have also $N[v_2]\subset N[v_1]$, which means that $v_1$ is a support of a twin, a contradiction.

Therefore, $v_2$ must belong to some ESI-set, and consequently $\es(G)>0$.

$(iii)\Rightarrow (i)$.
Assume $(iii)$ and suppose for a contradiction that $\es(G)>0$. Let $B$ be a nonempty ESI-set
and pick $x\in B$. By $(iii)$, the vertex $x$ is universal, or a twin, or a support of a twin,
contradicting Corollary~\ref{notESI1}. Hence $\es(G)=0$.
\end{proof}

Recall that every dominating set is isolating, and hence $\io(G)\le \gamma(G)$ for every graph $G$. It is therefore natural to ask for which graphs equality $\gamma(G)=\io(G)$ holds. In \cite{lemanska}, Lemańska et al.\ characterize trees for which $\gamma(G)=\iota(G)$ holds. Furthermore, in \cite{BDJK}, Brešar et al.\ identify several graph families for which the equality is satisfied.

Theorem~\ref{zero} describes graphs with $\es(G)=0$. Although this class is fairly large, the condition $\es(G)=0$ is not necessary for
$\gamma(G)=\io(G)$. Indeed, for the star $K_{1,n-1}$ we have $\gamma(K_{1,n-1})=\io(K_{1,n-1})=1$,
while $\es(K_{1,n-1})=n-2>0$ for $n>2$ by Proposition~\ref{families}.

Nevertheless, $\es(G)=0$ is a sufficient condition for $\gamma(G)=\io(G)$. In fact, Theorem~\ref{zero} $(ii)$ implies that every vertex $x$ has a neighbor $y\in N(x)$ with
$N[y]\subseteq N[x]$, which forces any isolating set to be dominating. Consequently, $\gamma(G)=\io(G)$. Combining this with Theorem~\ref{gen_bound} yields the following standard lower bound; see, e.g.,~\cite{core}.

\begin{corollary}
\label{gam=io}
If $G$ is a graph with $\es(G)=0$, then $\g(G)=\io(G)\geq \left\lceil\frac{n(G)}{\Delta(G)+1}\right\rceil$. 
\end{corollary}

Now we turn our attention to graphs attaining the upper bound in Theorem \ref{alpha_up_bound}, that is, to graphs for which $\es(G)=\al(G)-1$. Let $B$ be an ESI-set of a graph $G$ and let $B'\subseteq B$. The \emph{external private
neighborhood} of $B'$ with respect to $B$, denoted by $\EPN(B',B)$, is defined by
$$\EPN(B',B)=N(B')- N(B-B').$$
Equivalently, $\EPN(B',B)$ consists of all vertices that are adjacent to at least one vertex of $B'$ and to no vertex of $B-B'$.
\medskip

\begin{proposition}
\label{LB-clique}
Let $G$ be a graph  with at least one edge.  Then $\es(G)=\al(G)-1$ if and only if either $G=K_n$ for some $n\geq 2$, or for every $\es(G)$-set $B$ the following two conditions hold:
\begin{itemize}
\item[(i)] the set $D=V(G)-N[B]$ induces a complete graph, and
\item[(ii)] for every subset $B'\subseteq B$, $\alpha(G[\EPN(B',B)\cup D]) \leq |B'|+1$.
\end{itemize}
\end{proposition}

\begin{proof}
Let $G$ be a graph  with at least one edge. 
Assume first that $\es(G)=\alpha(G)-1$. The statement clearly holds if $G\cong K_n$, $n\geq 2$. Assume henceforth that $G$ is not complete and let $B$ be an $\es(G)$-set and $D=V(G)-N[B]$. Then $|B|=\es(G)=\alpha(G)-1$ and $D\neq\emptyset$. For every $v\in D$,
the set $B\cup\{v\}$ is independent (since $v\notin N[B]$) and has size $|B|+1=\al(G)$, and hence
it is an $\al(G)$-set. Now suppose for a contradiction that the set $D$ does not induce a complete graph. Then there exist two nonadjacent vertices $x,y\in D$ and $B\cup\{x,y\}$ is an independent set. Therefore, $\alpha(G) \ge |B|+2 =\alpha(G)+1$, a contradiction. Hence $D$ induces a complete graph, proving condition $(i)$.

To prove condition $(ii)$, assume to the contrary there exists an $\es(G)$-set $B$ with a subset $B'\subseteq B$ such that $\alpha(G[\EPN(B',B)\cup D]) > |B'|+1$. Let $H=G[\EPN(B',B)\cup D]$ and let $I$ be an $\alpha(H)$-set. Then $I\cup (B-B')$ is an independent set of $G$
\[\alpha(G)\geq |I\cup (B-B')|>|B'|+1+|B|-|B'|=|B|+1=\es(G)+1,\]
contradicting the assumption that $\es(G)=\alpha(G)-1$.

Now assume that conditions $(i)$ and $(ii)$ hold. 
Let $B$ be an $\es(G)$-set and put
$D=V(G)-N[B]$. To the contrary assume that $\es(G)\leq \alpha(G) -2$ and let $A$ be an $\alpha(G)$-set.
If $B\subseteq A$, then $|A-B|\ge 2$ and $A-B\subseteq D$ (since $A$ is
independent), so $D$ contains two nonadjacent vertices, contradicting~$(i)$. 
Therefore $B\not\subseteq A$. Let $A'=A-B$, $B_1=B\cap A$ and $B_2=B-A$. Since $\es(G)\leq \alpha(G) -2$, it follows that 
\[|B_1|+|A'|=|A|=\alpha(G)\geq \es(G)+2=|B_1|+|B_2|+2\]
and therefore $|A'|\geq |B_2|+2$. 
Furthermore, since $A$ is independent, no vertex of $A'$ is adjacent to any vertex of $B_1$.
Consequently, every vertex of $A'$ lies either in $D$ or in $N(B_2)-N(B_1)
=\EPN(B_2,B)$. That is, 
$A'\subseteq \EPN(B_2,B)\cup D$, 
and therefore $$
\al\left(G[\EPN(B_2,B)\cup D]\right)\ge |A'|\ge |B_2|+2,$$
contradicting condition $(ii)$ applied to $B'=B_2$. Therefore $\es(G)=\al(G)-1$, as required.
\end{proof}

The necessity of condition~$(ii)$ is illustrated by the graph $G$ shown in Figure~\ref{conditionii}. We have $\es(G)=2$, and $B=\{5,8\}$ is the unique $\es(G)$-set. Then $D=V(G)\setminus N[B]$ induces a complete graph, so condition~(i) holds. However, for $B'=\{5\}$, $\EPN(B',B)=\{4,6\}$ and
$$\alpha(G[\EPN(B',B)\cup D]) \ge 3 > |B'|+1=2,$$
so condition~$(ii)$ fails. On the other hand, $\alpha(G)=4$ and $\es(G)=2$, and hence $\es(G)\ne \alpha(G)-1$.

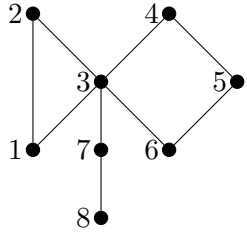
\begin{figure}[!ht]
\begin{center}
	\begin{tikzpicture}[scale=0.6]
		\tikzstyle{crn}=[circle,fill=black,draw, inner sep=0pt, minimum size=5pt]
			
		\begin{scope}
			\node (a)[crn] at (0 cm,0 cm){};
            \node (b)[crn] at (0 cm,3 cm){};
            \node (c)[crn] at (1.5 cm,1.5 cm){};
            \node (d)[crn] at (3 cm,3 cm){};
            \node (e)[crn] at (4.5 cm,1.5 cm){};
            \node (f)[crn] at (3 cm,0 cm){};
            \node (g)[crn] at (1.5 cm,0 cm){};
            \node (h)[crn] at (1.5 cm,-1.5 cm){};
            
            \draw (c)--(a)--(b)--(c)--(d)--(e)--(f)--(c);
            \draw (c)--(g)--(h);

            \draw[anchor = east] (a) node {$1$};
            \draw[anchor = east] (b) node {$2$};
            \draw[anchor = east] (c) node {$3$};
            \draw[anchor = east] (d) node {$4$};
            \draw[anchor = east] (e) node {$5$};
            \draw[anchor = east] (f) node {$6$};
            \draw[anchor = east] (g) node {$7$};
            \draw[anchor = east] (h) node {$8$};
		\end{scope}
    \end{tikzpicture}
    
\caption{A graph $G$ showing that condition~$(ii)$ is necessary.}
\label{conditionii}
\end{center}
\end{figure}

\medskip

Next, we present an upper bound that depends on $\delta(G), \Delta(G)$ and $n(G)$. Two direct consequences follows, one for $k$-regular graphs where $\Delta(G)=\delta(G)=k$ and one for trees where $\delta(T)=1$.\\

\begin{proposition}\label{prop:delta}
For any graph $G$, with $\Delta(G)\geq 2$, we have
\[\es(G) \leq \frac{\Delta(G)(\Delta(G)-1)}{\Delta(G)^2+\Delta(G)(\delta(G)-1)+\delta(G)}n(G)\]
and this bound is sharp.
\end{proposition}

\begin{proof}
Let $B$ be an $\es(G)$-set and let $D=V(G)-N[B]$. Every vertex in $B$ is adjacent to at least $\delta(G)$ vertices in $N(B)$, while every vertex in $N(B)$ is adjacent to at least one vertex in $D$ and therefore it has at most $\Delta(G)-1$ vertices in $B$. Furthermore, every vertex in $D$ is adjacent to at most $\Delta(G)$ vertices in $N(B)$. Counting the edges between $B$ and $N(B)$ gives $\delta(G)|B|\leq (\Delta(G)-1)|N(B)|$. Also, counting the edges between $N(B)$ and $D$, we have $|N(B)|\leq \Delta(G)|D|$.
It now follows that 
\begin{align*}
    n(G)&= |B|+|N(B)|+|D|\\
    &\geq |B| +|N(B)|(1+1/\Delta(G))\\
    &\geq |B|+\left (\frac{\Delta(G)+1}{\Delta(G)}\right )\left (\frac{\delta(G)}{\Delta(G)-1}\right )|B|\\
    & = \left (\frac{\Delta(G)(\Delta(G)-1)+\delta(G)(\Delta(G)+1)}{\Delta(G)(\Delta(G)-1)}\right )|B|.
\end{align*}
Since $B$ is an $\es(G)$-set, rearranging the above inequality yields the desired bound.

For sharpness, let $G=K_{m,m}$ with $m\ge 2$. Then $\delta(G)=\Delta(G)=m$ and $n(G)=2m$, so the
bound gives $\es(G)\le m-1$. Equality follows from Proposition~\ref{families}.
\end{proof}

\begin{corollary}
For a $k$-regular graph $G$ we have
\[\es(G) \leq \frac{k-1}{2k}n(G).\]
\end{corollary}

\begin{corollary}
For any tree $T\ncong K_1$ we have
\[\es(T) \leq \frac{\Delta(T)(\Delta(T)-1)}{\Delta(T)^2+1}n(T).\]
\end{corollary}

With the following construction we give a class of graphs attaining the bound in Proposition~\ref{prop:delta}. For $r<s$, these graphs satisfy $\delta(G)\ne \Delta(G)$.
Take $r\ge 1$ copies of $K_{1,s}$, for $s\ge 2$, and denote the center of the $k$-th copy by $c_k$ and its leaves by $b_{k,1},\dots,b_{k,s}$, where $k\in [r]$. 
For each $i\in [s]$, add $s-1$ new vertices $a_{i,1},\dots,a_{i,s-1}$, and join each vertex $a_{i,j}$, where $j\in [s-1]$, to all vertices $b_{k,i}$ for $k\in [r]$. 
Let the resulting graph be denoted by $G(r,s)$.
For $s\ge r$, we have $n(G(r,s))=r+rs+s(s-1)$, $\delta(G(r,s))=r$, and $\Delta(G(r,s))=s$. It follows from Proposition~\ref{prop:delta} that $\es(G(r,s))\le s(s-1)$.
Let $B=\{a_{i,j}: i\in [s], j\in [s-1]\}$. Then $B$ is an independent set and $V(G(r,s))- N[B]=\{c_1,\dots,c_r\}$. Since each $b_{k,i}$ is adjacent to $c_k$, it follows that $B$ is an ESI-set. Hence $\es(G(r,s))=s(s-1)$.
As a side effect, we also obtain that $\alpha(G(r,s))\ge \es(G(r,s))+r$, since $B\cup\{c_1,\dots,c_r\}$ is an independent set. In particular, $G_i=G(i,i)$ yields a family of graphs for which $\alpha(G_i)-\es(G_i)\ge i$, and hence the difference tends to infinity as $i\to\infty$.

\section{Trees}\label{tr}

In this section, we establish general upper and lower bounds for the externally supported independence number of trees, and characterize the trees for which these bounds are attained.

To prove the lower bound we use the following construction of an ESI-set $B$ for $T$, which we call a \emph{canonical ESI-set} of $T$. For each support vertex $w\in S(T)$ choose exactly one leaf
$t_w\in L(w)$, and define
$$B=L(T)-\{t_w: w\in S(T)\}.$$
Put $D=V(T)-N[B]$. Then $|B|=\ell(T)-s(T)$, and $B$ is an ESI-set: indeed, $B$ is
independent, and every vertex in $N(B)$ is a support of some leaf in $B$, hence it has a
neighbor in $D$, namely the chosen leaf adjacent to that support. In particular, both vertices of
$K_2$ are leaves and supports, so $\ell(K_2)-s(K_2)=0$.

Let $B$ be a canonical ESI-set and let $D=V(G)-N[B]$. The pair $(B,D)$ is called a \emph{canonical pair}. Define an operation called a \emph{switch at} $x\in S(T)$ as moving $x$ together with all its leaf neighbors to $D_x$, and moving the
chosen leaves $t_w$ (for each $w\in S(x)$) from $D$ to $B_x$. That is,
\[
B_x=(B- L(x))\cup \{t_w:w\in S(x)\}
\]
and
\[
D_x=\bigl(D-\{t_w:w\in S(x)\}\bigr)\cup\{x\}\cup L(x).
\]
Then $B_x$ is an ESI-set because
$D_x$ dominates $N(B_x)$.

Let $B'$ be an ESI-set and let $D'$ be a set dominating $N(B')$. Suppose that $x\in S(T)\cap D'$ and $t_x\in L(x)$. The operation called a \emph{reverse switch at} $x$ is the inverse of a switch at $x$, see Figure \ref{switch}. For each support vertex $w\in S(x)$ choose exactly one leaf
$t_w\in L(w)$, and define
 
\[
B^x=(B'-\{t_w:w\in S(x)\})\cup (L(x)-\{t_x\}) 
\]
and
\[
D^x=((D'-\{x\})-(L(x)-\{t_x\}))\cup \{t_w:w\in S(x)\}.
\]
Again, $D^x$ dominates $N(B^x)$ and thus $B^x$ is an ESI-set. 

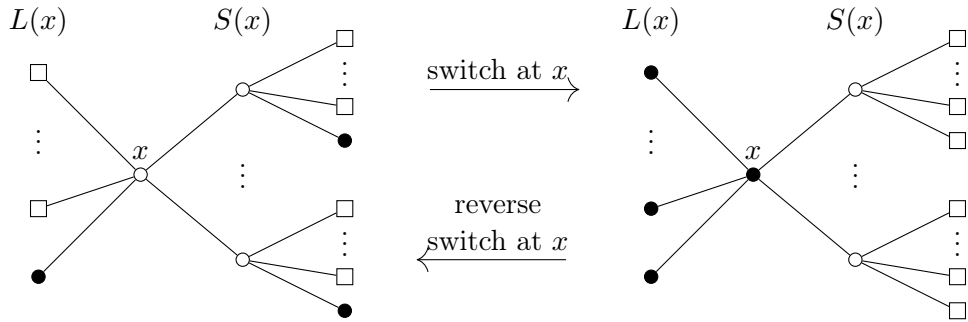
\begin{figure}[H]
\begin{center}
	\begin{tikzpicture}[scale=0.45]
		\tikzstyle{crn}=[circle,fill=black,draw, inner sep=0pt, minimum size=5pt]
        \tikzstyle{bel}=[circle,fill=white,draw, inner sep=0pt, minimum size=5pt]
        \tikzstyle{kvadrat}=[rectangle,fill=white,draw, inner sep=0pt, minimum size=6pt]
			
		\begin{scope}
            \node[draw=none, fill=none, inner sep=0pt] at (0 cm, 7.5 cm) {$L(x)$};
			\node (a)[kvadrat] at (0 cm,6 cm){};
            \node[draw=none, fill=none, inner sep=0pt] at (0 cm, 4.2 cm) {$\vdots$};
            \node (b)[kvadrat] at (0 cm,2 cm){};
            \node (c)[crn] at (0 cm,0 cm){};
		\end{scope}
			
		\begin{scope}[shift={(3,0)}]
			\node (x)[bel] at (0 cm,3 cm){};
            \draw (x) node[anchor=south, yshift=2pt] {$x$};
            \draw (x)--(a);
            \draw (x)--(b);
            \draw (x)--(c);
        \end{scope}

		\begin{scope}[shift={(6,0)}]
            \node[draw=none, fill=none, inner sep=0pt] at (0 cm, 7.5 cm) {$S(x)$};
			\node (a1)[bel] at (0 cm,5.5 cm){};
            \node[draw=none, fill=none, inner sep=0pt] at (0 cm, 3.2 cm) {$\vdots$};
            \node (b1)[bel] at (0 cm,0.5 cm){};
            \draw (x)--(a1);
            \draw (x)--(b1);
		\end{scope}

		\begin{scope}[shift={(9,0)}]
			\node (a2)[kvadrat] at (0 cm,7 cm){};
            \node[draw=none, fill=none, inner sep=0pt] at (0 cm, 6.2 cm) {$\vdots$};
            \node (b2)[kvadrat] at (0 cm,5 cm){};
            \node (c2)[crn] at (0 cm,4 cm){};
            \draw (a1)--(a2);
            \draw (a1)--(b2);
            \draw (a1)--(c2);
            \node (a3)[kvadrat] at (0 cm,2 cm){};
            \node[draw=none, fill=none, inner sep=0pt] at (0 cm, 1.2 cm) {$\vdots$};
            \node (b3)[kvadrat] at (0 cm,0 cm){};
            \node (c3)[crn] at (0 cm,-1 cm){};
            \draw (b1)--(a3);
            \draw (b1)--(b3);
            \draw (b1)--(c3);
		\end{scope}

        \begin{scope}[shift={(11.5,0)}]
            \draw[-{>[scale=1.8]}, shorten >=-5pt] (0,5.5) -- (4,5.5) node[midway, above] {switch at $x$};
            \draw[{<[scale=1.8]}-, shorten <=-5pt] (0,0.5) -- (4,0.5) node[midway, above, align=center] {reverse \\ switch at $x$};
        \end{scope}

		\begin{scope}[shift={(18,0)}]
            \node[draw=none, fill=none, inner sep=0pt] at (0 cm, 7.5 cm) {$L(x)$};
			\node (aa)[crn] at (0 cm,6 cm){};
            \node[draw=none, fill=none, inner sep=0pt] at (0 cm, 4.2 cm) {$\vdots$};
            \node (bb)[crn] at (0 cm,2 cm){};
            \node (cc)[crn] at (0 cm,0 cm){};
		\end{scope}
			
		\begin{scope}[shift={(21,0)}]
			\node (y)[crn] at (0 cm,3 cm){};
            \draw (y) node[anchor=south, yshift=2pt] {$x$};
            \draw (y)--(aa);
            \draw (y)--(bb);
            \draw (y)--(cc);
        \end{scope}

		\begin{scope}[shift={(24,0)}]
            \node[draw=none, fill=none, inner sep=0pt] at (0 cm, 7.5 cm) {$S(x)$};
			\node (aa1)[bel] at (0 cm,5.5 cm){};
            \node[draw=none, fill=none, inner sep=0pt] at (0 cm, 3.2 cm) {$\vdots$};
            \node (bb1)[bel] at (0 cm,0.5 cm){};
            \draw (y)--(aa1);
            \draw (y)--(bb1);
		\end{scope}

		\begin{scope}[shift={(27,0)}]
			\node (aa2)[kvadrat] at (0 cm,7 cm){};
            \node[draw=none, fill=none, inner sep=0pt] at (0 cm, 6.2 cm) {$\vdots$};
            \node (bb2)[kvadrat] at (0 cm,5 cm){};
            \node (cc2)[kvadrat] at (0 cm,4 cm){};
            \draw (aa1)--(aa2);
            \draw (aa1)--(bb2);
            \draw (aa1)--(cc2);
            \node (aa3)[kvadrat] at (0 cm,2 cm){};
            \node[draw=none, fill=none, inner sep=0pt] at (0 cm, 1.2 cm) {$\vdots$};
            \node (bb3)[kvadrat] at (0 cm,0 cm){};
            \node (cc3)[kvadrat] at (0 cm,-1 cm){};
            \draw (bb1)--(aa3);
            \draw (bb1)--(bb3);
            \draw (bb1)--(cc3);
		\end{scope}

    \end{tikzpicture}
    
\caption{An example illustrating a switch at $x$ and a reverse switch at $x$. White square vertices belong to an ESI-set $X$ and black vertices belong to $V(G)-N[X]$.}
\label{switch}
\end{center}
\end{figure}

\begin{theorem}
\label{tree-lower}
If $T$ is a tree, then $\es(T)\geq \ell(T)-s(T)$, where equality holds if and only if $V(T)=S(T)\cup L(T)$ and $s(v) \leq \ell(v)-1$ for every $v\in S(T)$.
\end{theorem}

\begin{proof}
Let $T$ be a tree and let $B$ be its canonical ESI-set.  So,  $|B|=\ell(T)-s(T)$ and the lower bound $\es(T)\geq |B|=\ell(T)-s(T)$ follows. 
We next characterize when the equality holds.

We first prove that if $V(T)-(S(T)\cup L(T))\neq \emptyset$, then $\es(T)>\ell(T)-s(T)$.

Let $v\in V(T)-(S(T)\cup L(T))$ and let $z$ be a leaf at minimum distance from $v$ with $v=v_0,v_1,\dots,v_k=z$ the unique $v$--$z$ path. Then $v_{k-1}\in S(T)$. Since $v$ is not a support vertex, we have $k\ge 2$, and hence $u=v_{k-2}$ is a non-leaf vertex with $\ell(u)=0$ and $w=v_{k-1}\in S(u)$. 
In a canonical ESI-set $B$, the chosen leaf $t_w$ satisfies $t_w\notin B$. Define $B^\star=B\cup\{t_w\}$, which is clearly an independent set. Then we have $|B^\star|=|B|+1$. Let $D^\star=V(T)-N[B^\star]$. Since $u$ is not adjacent to any leaf, it is not adjacent to any vertex of $B^\star$, and thus $u\in D^\star$.
We now verify the ES-condition for $B^\star$. Every vertex of $N(B)$ keeps an external neighbor in $D^\star$ as in the canonical construction. Moreover, the only vertex that may be newly added to $N(B^\star)$ is $w$, 
which has a neighbor in $D^\star$, namely $u$. Hence $B^\star$ is an ESI-set and therefore $\es(T)\ge |B^\star|=|B|+1>|B|=\ell(T)-s(T)$.

Now we prove that the existence of a support vertex $v$ with $s(v)>\ell(v)-1$ forces the strict inequality $\es(T)>\ell(T)-s(T)$. Let $v\in S(T)$ with $s(v)>\ell(v)-1$ and perform a switch at $v$ on the canonical pair $(B,D)$ to obtain $(B_v,D_v)$.
Since $B\cap L(v)=L(v)-\{t_v\}$, we remove $\ell(v)-1$ leaves from $B$ and add $s(v)$
leaves (the vertices $t_w$ for every $w\in S(v)$) to obtain $B_v$. Therefore
$|B_v|=|B|-(\ell(v)-1)+s(v) > |B|$, since $s(v)>\ell(v)-1$, 
and hence $\es(T)\ge |B_v|>|B|=\ell(T)-s(T)$. 

Conversely, assume that $V(T)=S(T)\cup L(T)$ and $s(v)\le \ell(v)-1$ holds for every $v\in S(T)$.
Suppose for a contradiction that $\es(T)>\ell(T)-s(T)$, and let $B$ be an $\es(T)$-set with
$D=V(T)-N[B]$. Among all $\es(T)$-sets choose $B$ so that $|S(T)\cap D|$ is minimum.

If $S(T)\cap D\neq\emptyset$, pick $x\in S(T)\cap D$ and perform a reverse switch at $x$,
obtaining an ESI-set $B^{x}$. By construction,
$|B^{x}|=|B|-s(x)+(\ell(x)-1)\ \ge\ |B|$. Since $|B|=\es(T)$, we have $|B^{x}|=|B|$, while the reverse
switch removes $x$ from $D\cap S(T)$, contradicting the choice of $B$. Hence
$S(T)\cap D=\emptyset$.

Since $V(T)=S(T)\cup L(T)$ and $S(T)\cap D=\emptyset$, it follows that
$D\subseteq L(T)$, i.e.~every vertex of $D$ is a leaf.

We may further assume that $B\subseteq L(T)$. Indeed, if $x\in B$ were a non-leaf vertex, then $x\in S(T)$,
contradicting Corollary~\ref{notESI1}, which states that an ESI-set contains no support vertex.

Now fix a support vertex $w\in S(T)$. If $B$ contains all leaves of $L(w)$, then $w\in N(B)$,
and the ES-condition requires $w$ to have a neighbor in $D$, a contradiction.

Therefore $|B\cap L(w)|\le \ell(w)-1$ for every $w\in S(T)$, and hence
$$|B|
=\sum_{w\in S(T)}|B\cap L(w)|
\le \sum_{w\in S(T)}(\ell(w)-1)
=\ell(T)-s(T),$$
contradicting $|B|=\es(T)>\ell(T)-s(T)$. 
\end{proof}

\medskip

Now we turn our attention to the upper bound. If $T\ncong K_1$ is a tree, then Theorem~\ref{alpha_up_bound} yields $\es(T)\le \alpha(T)-1$. In what follows we characterize trees that attain this bound.
To do so we will need the following special families of trees,
depicted in Figure \ref{abcd}.

\begin{figure}[!ht]
\begin{center}
	\begin{tikzpicture}[scale=0.5]
		\tikzstyle{crn}=[circle,fill=black,draw, inner sep=0pt, minimum size=5pt]

        
\node [crn, name=b]   at (1,0) {};
\node [crn, name=c]   at (3,0) {};
\node [crn, name=d]   at (2,2) {};
\node[font=\large\bfseries] at (2,0.5){$\ldots$}; 
 
\node [crn, name=x, label=above:$x$]   at (4,4) {};

 \draw (d)--(b);
 \draw (x)--(d)--(c);

 \node [crn, name=b1]   at (5,0) {};
 \node [crn, name=c1]   at (7,0) {};
 \node [crn, name=d1]   at (6,2) {};
\node[font=\large\bfseries] at (6,0.5){$\ldots$};

 \draw (d1)--(b1);
 \draw (x)--(d1)--(c1);

\node[font=\large\bfseries] at (4,2){$\ldots$};

\node[font=\large\bfseries] at (4,-1.5){$T_1 \in \mathcal{A}$};

        
\node [crn, name=b2]   at (9,0) {};
 \node [crn, name=c2]   at (11,0) {};
 \node [crn, name=d2]   at (10,2) {};
\node[font=\large\bfseries] at (10,0.5){$\ldots$}; 
 
 \node [crn, name=x2, label=left:$x$]   at (12,4) {};

 \draw (d2)--(b2);
 \draw (x2)--(d2)--(c2);

 \node [crn, name=b3]   at (13,0) {};
 \node [crn, name=c3]   at (15,0) {};
 \node [crn, name=d3]   at (14,2) {};
\node[font=\large\bfseries] at (14,0.5){$\ldots$};

 \draw (d3)--(b3);
 \draw (x2)--(d3)--(c3);

\node[font=\large\bfseries] at (12,2){$\ldots$};

\node [crn, name=y2, label=left:$y$]   at (12,6) {};
\draw (x2)--(y2);

\node[font=\large\bfseries] at (12,-1.5){$T_2 \in \mathcal{B}$};



\node [crn, name=b4]   at (17,0) {};
\node [crn, name=c4]   at (19,0) {};
\node [crn, name=d4]   at (18,2) {};
\node[font=\large\bfseries] at (18,0.5){$\ldots$}; 
 
\node [crn, name=x4, label=above:$x$]   at (20,4) {};

\draw (d4)--(b4);
\draw (x4)--(d4)--(c4);
 
\node [crn, name=b5]   at (21,0) {};
\node [crn, name=c5]   at (23,0) {};
\node [crn, name=d5]   at (22,2) {};
\node[font=\large\bfseries] at (22,0.5){$\ldots$}; 
 
\draw (d5)--(b5);
\draw (x4)--(d5)--(c5);

\node[font=\large\bfseries] at (20,2){$\ldots$};


\node [crn, name=b5]   at (25,0) {};
\node [crn, name=c5]   at (27,0) {};
\node [crn, name=d5]   at (26,2) {};
\node[font=\large\bfseries] at (26,0.5){$\ldots$}; 
 
\node [crn, name=y5, label=above:$y$]   at (28,4) {};

\draw (d5)--(b5);
\draw (y5)--(d5)--(c5);
 
\node [crn, name=b6]   at (29,0) {};
\node [crn, name=c6]   at (31,0) {};
\node [crn, name=d6]   at (30,2) {};
\node[font=\large\bfseries] at (30,0.5){$\ldots$}; 
 
\draw (d6)--(b6);
\draw (y5)--(d6)--(c6);

\node[font=\large\bfseries] at (28,2){$\ldots$};

\draw (x4)--(y5);

\node[font=\large\bfseries] at (24,-1.5){$T_3 \in \mathcal{C}$};

\end{tikzpicture}
\vspace{-15pt}
\caption{Trees in $\mathcal{A}\cup \mathcal{B}\cup \mathcal{C}$.}
\label{abcd}
\end{center}
\end{figure}
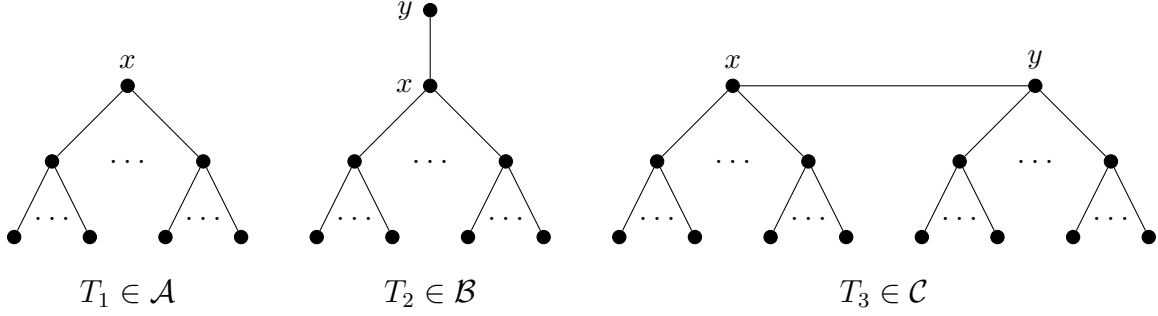

Let $\mathcal{A}$ be the family of trees constructed as follows.
Start with a star $K_{1,t}$, where $t\ge 2$, with center $x$.
For each leaf $v$ of this star, attach $k_v\ge 1$ new leaves to $v$.

To construct a tree $T$ in the family $\mathcal{B}$, 
start with a star $K_{1,t}$, where $t\ge 2$, with center $x$. For each leaf $v$ with the exception of precisely one leaf $y$ of this star, attach $k_v\ge 1$ new leaves to $v$ (i.e., among the neighbors of $x$, only $y$ remains a leaf in $T$).

Finally, let $\mathcal{C}$ be the family of trees constructed as follows.
Start with a \emph{double star} $S_{a,b}$, where $a,b\ge 1$, that is,
the tree obtained from an edge $xy$ by attaching $a$ leaves to $x$
and $b$ leaves to $y$.
For each leaf $v$ of $S_{a,b}$, attach $k_v\ge 1$ new leaves to $v$.

We first compute the independence number of trees in $\mathcal{A}$, $\mathcal{B}$ and $\mathcal{C}$.
We will use the following result.

\begin{lemma}
\cite{martinez}
\label{mar}
If $T$ is a tree of order $n\geq 3$, then there exists an $\al(T)$-set containing all the leaves of $T$.
\end{lemma}

\begin{lemma}
\label{alfa-A}
Let $T\in\mathcal{A}\cup \mathcal{B}$, and let $x$ be the center of the initial star in the construction of $T$. If $B$ is the set of all vertices at distance $2$ from $x$, then $\alpha(T)=|B|+1$.
\end{lemma}

\begin{proof} Let $T\in\mathcal{A}\cup \mathcal{B}$, and let $x$ be the central vertex of the initial star in the construction of $T$. By Lemma \ref{mar}, there exists an $\al(T)$-set $I$ containing all the leaves of $T$.

If $T\in\mathcal{B}$, the leaves of $T$ are precisely $y$ together with 
the vertices at distance $2$ from $x$, that is, $L(T)=B\cup\{y\}$. Every vertex in $V(T)-L(T)$ is a support of a leaf, hence no such vertex can belong to $I$. Therefore 
$I=L(T)$, and consequently,
$\al(T)=|I|=|L(T)|=|B|+1$.

Now let $T\in\mathcal{A}$.
All leaves of $T$ are at distance $2$ from $x$, so $L(T)=B$. Thus $B\subseteq I$.
Moreover, $x$ is not adjacent to any vertex of $B$, so $B\cup\{x\}$ is independent and has
size $|B|+1$. Since every neighbor of $x$ is the support of at least one leaf in $B$, no
neighbor of $x$ can be added to an independent set containing $B$. Hence $B\cup\{x\}$ is
maximum and $\al(T)=|B|+1$.

\end{proof}

\begin{lemma}
\label{alfa-B}
If $T\in\mathcal{C}$ and $B$ is the set of all leaves of $T$, then $\alpha(T)=|B|+1$.
\end{lemma}

\begin{proof} 
Let $T\in\mathcal{C}$ and let $xy$ be the central edge of the underlying double star $S_{a,b}$.
Let $B=L(T)$ be the set of all leaves of $T$. By Lemma~\ref{mar}, there exists an $\al(T)$-set
$I$ containing $B$.
Every vertex of $V(T)-(B\cup\{x,y\})$ is a support vertex of a leaf from $B$ and therefore cannot belong to $I$. Moreover, neither $x$ nor $y$ is
adjacent to any vertex of $B$, so both $B\cup\{x\}$ and $B\cup\{y\}$ are independent. Since
$x$ and $y$ are adjacent, no independent set containing $B$ can contain both of them.
Therefore, both $B\cup\{x\}$ and $B\cup\{y\}$ are maximum independent sets, and hence $\al(T)=|B|+1$.
\end{proof}

\begin{theorem}
\label{tree-up}
Let $T$ be a nontrivial tree. Then $\es(T)=\alpha(T)-1$ if and only if $T$ is a star $K_{1,t}$, where $t\ge 1$, or $T\in \mathcal{A}\cup \mathcal{B} \cup \mathcal{C}$.
\end{theorem}

\begin{proof}
If $T$ is a star $K_{1,t}$, where $t\ge 1$, then $\al(T)=t$ and $\es(T)=t-1$ by Proposition \ref{families}, so $\es(T)=\al(T)-1$.

Now assume $T\in\mathcal{A}\cup\mathcal{B}$ and let $x$ be the center of the initial star in the construction of $T$. If $B$ is the set of vertices at distance $2$ from $x$, then $B$ is independent. Moreover, every vertex in $N(B)$ is adjacent to $x$, and hence has a neighbor in
$D=V(T)-N[B]$ (indeed, $x\in D$). Thus $B$ is an ESI-set. By Lemma~\ref{alfa-A},
$\al(T)=|B|+1$, so $|B|=\al(T)-1$. Therefore $\es(T)\ge |B|=\al(T)-1$, and together with
Theorem~\ref{alpha_up_bound} we obtain $\es(T)=\al(T)-1$.

Finally, let $T\in\mathcal{C}$ and let $B$ be the set of all leaves of $T$. Then $B$ is independent.
Every neighbor of a vertex in $B$ is a support vertex adjacent to $x$ or to $y$, hence it has a
neighbor in $D=\{x,y\}\subseteq V(T)-N[B]$. Thus $B$ is an ESI-set and $\es(T)\ge |B|$.
By Lemma~\ref{alfa-B} we have $|B|=\al(T)-1$, so again Theorem~\ref{alpha_up_bound} implies
$\es(T)=\al(T)-1$.

\medskip

To prove the converse, assume that $T$ is a nontrivial tree with $\es(T)=\alpha(T)-1$.
If $\alpha(T)=1$, then $T= K_2$, and we are done. 
Hence assume $\alpha(T)\ge 2$ and fix an $\es(T)$-set $B$. By Proposition~\ref{LB-clique}, the subgraph $W$ induced by $D=V(T)-N[B]$ is complete.
Since $T$ is a tree, this implies  $W\in \{K_1,K_2\}$.

First we prove the claim that every component of $T-W$ is a star whose center is the unique vertex adjacent to a vertex of $W$.
Let $C$ be a component of $T-W$. Because $T$ is a tree and $W$ is connected, $C$ has a unique vertex $c\in V(C)$ adjacent to precisely one vertex of $W$. We cannot have $C=K_1$, since $T$ is connected and every vertex adjacent to a vertex of $W$ must have a neighbor in $B$.
Hence there exists a vertex in $V(C)-\{c\}$.
Let $u\in V(C)-\{c\}$ be arbitrary. Then $u\notin W$, so $u\in N[B]=B\cup N(B)$.
If $u\in N(B)$, then the ES-condition forces $u$ to have a neighbor in $W$,
contradicting the fact that $c$ is the unique vertex of $C$ adjacent to a vertex in $W$.
Hence $u\notin N(B)$ and therefore $u\in B$. Since $u$ was arbitrary, we have $V(C)-\{c\}\subseteq B$.
Since $B$ is independent, there are no edges among vertices of $V(C)-\{c\}$, so every edge of $C$ is incident with $c$.
Thus $C$ is a star centered at $c$, as claimed. We now distinguish two possible cases.

\smallskip\noindent
\emph{Case 1: $W=K_1$.}
Let $V(W)=\{x\}$. Then, by the claim above, every component of $T-W$ is a star of order at least 2 whose center is the unique neighbor of $x$ in that component.

If $T-W$ has exactly one component, $T$ is a star. Otherwise every neighbor of $x$ is the center of a star component of order at least $2$, and hence $T\in\mathcal{A}$.

\smallskip\noindent
\emph{Case 2: $W=K_2$.} Let $W$ contain adjacent vertices $x$ and $y$.
By the above claim, every component of $T-W$ is a star of order at least 2 attached to exactly one of $x$ and $y$. If both $x$ and $y$ have a
neighbor outside $W$, then $T\in\mathcal{C}$. On the other hand, if only one of $x$ and $y$ has a neighbor outside $W$, then $T\in\mathcal{B}$.
\end{proof}

Since $K_2$ can be viewed as a star, and $P_4 \in \mathcal{B}$,\, $P_3, P_5\in \mathcal{A}$ and  $P_6\in \mathcal{C}$, Theorem \ref{tree-up} implies the following.

\begin{corollary}
For $n\ge 2$, $\es(P_n)=\alpha(P_n)-1$ if and only if $2\le n\le 6$.
\end{corollary}

Theorem \ref{tree-up}  also shows that the highly centralized structure of a star is optimal for maximizing the externally supported independence number among trees of fixed order $n$. In other words, if $T$ is a tree, then $\es(T)\le n-2$, with equality if and only if $T$ is a star.

\medskip

\section{Linear algorithm for trees}

From an algorithmic point of view, observe that the \textsc{ESI Set} problem can be expressed in monadic second-order logic (MSO). Therefore, Courcelle's theorem \cite{Cour} implies the existence of a linear-time algorithm for graphs of bounded treewidth. Furthermore, the Courcelle--Makowsky--Rotics theorem \cite{CoMR} guarantees a polynomial-time algorithm for graphs of bounded clique-width. Since trees have treewidth~1, the above result already implies the existence of a linear-time algorithm for trees. In the remainder of this section, we present such an algorithm explicitly.

Let $T$ be a rooted tree with root $r$.
For each vertex $v\in V(T)$, denote by $Ch(v)$ its set of children and by $T_v$ the subtree of $T$ rooted at $v$.
Let $B\subseteq V(T)$ be an $\es(T)$-set.
Each vertex is assigned exactly one type: $S$ if $v\in B$, $N$ if $v\in N(B)$, and $R$ if $v\in V(T)-N[B]$. 

For every vertex $v$, type $x\in\{S,N,R\}$ and parent type
$p\in\{S,N,R\}$ we define $DP[v][x][p]$
to be the maximum number of vertices of type $S$ in the subtree $T_v$, subject to $v$ having type $x$ and its parent having type $p$. Impossible states take value $-\infty$.

The algorithm processes the vertices of $T$ in postorder, that is, after all their children have been processed.
For each vertex $v$, the values corresponding to all six possible states $DP[v][x][p]$ are computed from the previously computed values of its children. The ESI-number of $T$ is then obtained from the possible states of the root.



\begin{algorithm}[H]
\caption{ESI-number of a tree}
\begin{algorithmic}[0]
\Require rooted tree $T$ with root $r$
\Ensure $\es(T)$

\State compute a postorder ordering of the vertices

\ForAll{$v$ in postorder}

    \Statex \textbf{(A) $v$ has type $S$}
    \State $$ DP[v][S][N] \gets
           1 + \sum_{u\in Ch(v)} DP[u][N][S]$$

    \Statex \textbf{(B) $v$ has type $R$}
    \ForAll{$p \in \{N,R\}$}
        \State $$ DP[v][R][p] \gets
        \sum_{u\in Ch(v)}
        \max\{DP[u][N][R], DP[u][R][R]\}$$
    \EndFor

    \Statex \textbf{(C) $v$ has type $N$}
    \ForAll{$p \in \{S,N,R\}$}
        \State compute
        $$
        DP[v][N][p] =
        \max \sum_{u\in Ch(v)} DP[u][x_u][N]
        $$
        \hspace{1.2cm}over assignments $x_u\in\{S,N,R\}$ such that
        \begin{itemize}[leftmargin=1.7cm]
            \item if $p\neq S$, then at least one child has $x_u=S$,
            \item if $p\neq R$, then at least one child has $x_u=R$.
        \end{itemize}
    \EndFor

\EndFor

\State \Return $\max\{DP[r][S][N], DP[r][N][N], DP[r][R][N]\}$
\end{algorithmic}
\end{algorithm}

\begin{theorem}
If $T$ is a rooted tree with root $r$, then Algorithm ESI-number of a tree computes $\es(T)$ in $O(n(T))$ time and $O(n(T))$ space.
\end{theorem}

\begin{proof}
For every vertex $v\in V(T)$, $x\in\{S,N,R\}$ and $p\in\{S,N,R\}$, the entry $DP[v][x][p]$ is defined as the maximum number of vertices of type $S$ in the subtree $T_v$, subject to $v$ having type $x$ and the parent of $v$ having type $p$; impossible states take value $-\infty$. Since the vertices are processed in postorder (that is, after all their children have been processed), when $v$ is processed all values $DP[u][\cdot][\cdot]$ for children $u\in Ch(v)$ are already computed. We next justify the three update rules. \\

\noindent\textbf{(A) $v$ has type $S$.} \\ 
If $v$ has type $S$, then $v\in B$ and contributes $1$ to the objective. Moreover, since $v$ has type $S$ every child $u\in Ch(v)$ must have type $N$. Therefore, the optimal value in $T_v$ equals 
$DP[v][S][N] \;=\; 1 + \sum_{u\in Ch(v)} DP[u][N][S],$
which is exactly the update rule in case~(A). \\

\noindent\textbf{(B) $v$ has type $R$.} \\
If $v$ has type $R$, then $v\notin N[B]$ and the parent of $v$ is not of type $S$. Hence, only $p\in\{N,R\}$ is possible. Furthermore, no child of $v$ can be of type $S$, so each child $u\in Ch(v)$ may only have type $N$ or $R$. Since the subtrees $T_u$ are vertex-disjoint, the optimal choice is obtained by selecting independently, for each child $u$, the better of the two possible states. Consequently, for $p\in\{N,R\}$ we have
$DP[v][R][p] \;=\; \sum_{u\in Ch(v)} \max\{DP[u][N][R],\,DP[u][R][R]\},$
which is the update rule in case~(B). \\

\noindent\textbf{(C) $v$ has type $N$.} \\
If $v$ has type $N$, then $v\in N(B)$ and it must be adjacent to at least one vertex of type $S$. If $p=S$, this requirement is already satisfied; otherwise (if $p\neq S$) at least one child of $v$ must have type $S$. 
Furthermore, ES-condition imposes that if the parent of $v$ is not of type $R$ (if $p\neq R$), then at least one child of $v$ must have type $R$. If $p=R$, then this condition is already satisfied.
For a fixed assignment of types $x_u\in\{S,N,R\}$ to all children $u\in Ch(v)$ the best achievable value in $T_u$ equals $DP[u][x_u][N]$. Since the subtrees $T_u$ are disjoint, the total number of vertices of type $S$ in $T_v$ equals $\sum_{u\in Ch(v)} DP[u][x_u][N]$. Maximizing this sum over all assignments $(x_u)_{u\in Ch(v)}$ that satisfy the above two constraints yields exactly the update rule in case~(C). \\

Since cases (A)--(C) exhaust all possibilities for the type of $v$, it follows that, for every vertex $v$, the computed entries $DP[v][x][p]$ equal the optimal number of vertices of type $S$ in $T_v$ under the stated boundary conditions. In particular, the optimal value for the whole tree is obtained at the root $r$. 
Since the root has no parent, we evaluate it with parent type $N$. In this way, all requirements involving the parent of the root must be satisfied within the subtree rooted at $r$.
The three possible types of the root are $S$, $N$, and $R$, corresponding to the states $DP[r][S][N]$, $DP[r][N][N]$, and $DP[r][R][N]$, respectively.
Therefore, the returned value 
$\max\{DP[r][S][N],\,DP[r][N][N],\,DP[r][R][N]\}$ is equal to $\es(T)$. This completes the proof of correctness. \\ 

For the time complexity, observe that we process each vertex of the tree exactly once in postorder. For every vertex $v$, we consider only six possible states $(x,p)$, namely $(S,N)$, $(R,N)$, $(R,R)$, $(N,S)$, $(N,N)$, and $(N,R)$. 
In case (C), it suffices to know whether a child of type $S$ has been selected and whether a child of type $R$ has been selected. 
Therefore, only four states need to be maintained, depending on whether neither requirement, only the requirement of having a child of type $S$, only the requirement of having a child of type $R$, or both requirements have already been satisfied.
When processing a child, we update a constant number of states, and therefore each child requires $O(1)$ time. Hence, the total time spent at vertex $v$ is $O(|Ch(v)|)$.
Since $\sum_{v \in V(T)} |Ch(v)| = |E(T)| = n(T) - 1$, the overall running time is $O(n(T))$. The table $DP$ stores $O(1)$ values per vertex, therefore the space usage is $O(n(T))$.

\end{proof}



\section{Conclusion}

In this work we initiated the study of $\es(G)$, which improves the lower bound for $\io(G)$ as presented in Theorem \ref{low_bound}. We showed that the ESI set problem is NP-complete for general graphs and presented a polynomial-time algorithm for $\es(T)$ for every tree $T$. It would be interesting to better understand the boundary between graph classes for which $\es(G)$ can be computed in polynomial time and those for which the problem is NP-complete.

\begin{problem}
Identify graph classes where it is NP-hard to derive $\es(G)$ or where there exists a polynomial algorithm to determine $\es(G)$.  
\end{problem}

Several additional connections of $\es(G)$ with other graph parameters can be expected. We mention two such connections. Let $G$ be a graph and $f:V(G)\rightarrow \{0,1,2\}$ a map. If every vertex $v\in V(G)$ with $f(v)=0$ has a neighbor $u$ with $f(u)=2$, then $f$ is a \emph{Roman domination function} of $G$. The \emph{Roman domination number} $\g_R(G)$ of $G$ is the minimum weight $w(f)=\sum_{v\in V(G)}f(v)$ over all Roman domination functions $f$ of $G$, see \cite{CDHH}. Now, if $B$ is an ESI-set of $G$ together with a dominating set $D$ of $N(B)$, then we can define a function
\begin{equation*}
f(v)=\left\{\begin{array}{ll}
0 & \mbox{if } v\in N(B), \\[0.15cm]
1 & \mbox{if } v\in B, \\[0.15cm]
2 & \mbox{if } v\in D, 
\end{array}\right.
\end{equation*}
which is clearly a Roman dominating function as every vertex from $N(B)$ is dominated by $D$. (Notice also that we can exchange $B$ and $D$ in the above function and we still obtain a Roman dominating function as every vertex from $N(B)$ must have a neighbor in $B$ as well.) Moreover, this Roman dominating function has additional properties that vertices of weight one form an independent set and there are no edges between vertices of weight one and weight two. This may be closer to independent Roman domination where in addition there are also no edges between vertices of weight two, see the seminal work \cite{ATRM}.  Nevertheless, it seems an interesting approach for possible connections between graphs where $\es(G)$-set yields an optimal (independent) Roman dominating function. 

\begin{problem}
Describe all graphs where an $\es(G)$-set yields an optimal (independent) Roman dominating function.
\end{problem}

The other parameter is a packing, more precisely, a $3$-packing of a graph $G$. Recall that a set $P\subseteq V(G)$ forms a $k$-\emph{packing} of $G$ if the distance between any two vertices of $P$ is more than $k$ and $k$-\emph{packing number} $\rho_k$ of $G$ is the maximum cardinality of a $k$-packing of $G$, see \cite{BKPY} and the references therein. Clearly, $\rho_1(G)=\al(G)$, but here we are more interested in $\rho_3(G)$. Let $P$ be a $3$-packing of $G$. Notice first that if $v\in P$ is a support of a leaf $u$, then $P'=(P\cup\{u\})-\{v\}$ is a $3$-packing as well and we may choose a $3$-packing $P$ without any supports. Further, if $u,v\in P$ are at distance four, then the neighbors of $u$ and $v$ that belong to a shortest $u,v$-path have a neighbor on this path outside of $N[P]$. So, a $3$-packing $P$ is an ESI-set if for every $v\in P$ and every $u\in N(v)$ we have $N(u)-N[v]\neq \emptyset$. Whence the following problem.    

\begin{problem}
Describe all graphs where $\es(G)=\rho_3(G)$.
\end{problem}

Let us conclude with the following generalization obtained from $\io_{\mathcal{F}}(G)$ where $\mathcal{F}\neq\{K_2\}$. In our case of $\es(G)$ we described the maximum cardinality of an independent set $B$ of vertices such that $V(G)-N[B]$ dominates $N(B)$. To obtain a generalization, we need to consider a set of vertices $B_{\mathcal{F}}$ such that no graph from $\mathcal{F}$ is contained in an induced subgraph of $G$ on vertices from $B_{\mathcal{F}}$ and that every vertex from $N[B_{\mathcal{F}}]-B_{\mathcal{F}}$ is dominated by $V(G)-N[B_{\mathcal{F}}]$. Such an approach yields an upper bound for $\io_{\mathcal{F}}(G)$ similar to the one for $\io(G)$ in Theorem \ref{gen_bound}.   

\vskip 1pc \noindent{\bf Acknowledgments.} 
This work has been supported by the European Commission's Horizon Europe Research and Innovation programme through the Marie Skłodowska-Curie Actions Staff Exchanges (MSCA-SE) under Grant Agreement no.101182819 (COVER: (C)ombinatorial (O)ptimization for (V)ersatile Applications to (E)merging u(R)ban Problems).
I.P. was partially supported by the Slovenian Research and Innovation Agency (ARIS), program no.~P1--0297. 
A. T. was partially supported by the Slovenian Research and Innovation Agency (ARIS), program no.~P1--0297 and project no.~J1--70016. 
D.B. was partially supported by the Slovenian Research and Innovation Agency (ARIS), program no.~P2--0065 and project no.~J1--70016.\\
We would also like to thank M. Milani\v c for his helpful comments and suggestions regarding the algorithmic aspects of the paper.

\vskip 1pc \noindent{\bf Statements and Declarations.}
Authors have no conflict of interest to declare.

\bibliographystyle{plain}
\bibliography{references}

@article {ATRM,
    AUTHOR = {Adabi, M. and Targhi, E.E. and Jafari Rad, N. and
              Moradi, M.S.},
     TITLE = {Properties of independent {R}oman domination in graphs},
   JOURNAL = {Australas. J. Combin.},
  FJOURNAL = {The Australasian Journal of Combinatorics},
    VOLUME = {52},
      YEAR = {2012},
     PAGES = {11--18},
}

@article {BaBS,
    AUTHOR = {Bartolo, K. and Borg, P. and Scicluna, D.},
     TITLE = {Isolation of squares in graphs},
   JOURNAL = {Discrete Math.},
    VOLUME = {347},
      YEAR = {2024},
    NUMBER = {12},
     PAGES = {Paper No. 114161, 11},
       DOI = {10.1016/j.disc.2024.114161},
       URL = {https://doi.org/10.1016/j.disc.2024.114161},
}

@article {Borg,
    AUTHOR = {Borg, P.},
     TITLE = {Isolation of cycles},
   JOURNAL = {Graphs Combin.},
    VOLUME = {36},
      YEAR = {2020},
    NUMBER = {3},
     PAGES = {631--637},
       DOI = {10.1007/s00373-020-02143-2},
       URL = {https://doi.org/10.1007/s00373-020-02143-2},
}

@article {BoGo,
    AUTHOR = {Boyer, G. and Goddard, W.},
     TITLE = {Disjoint isolating sets and graphs with maximum isolation number},
   JOURNAL = {Discrete Appl. Math.},
    VOLUME = {356},
      YEAR = {2024},
     PAGES = {110--116},
       DOI = {10.1016/j.dam.2024.05.022},
       URL = {https://doi.org/10.1016/j.dam.2024.05.022},
}

@article {BoGo2,
    AUTHOR = {Boyer, G. and Goddard, W.},
     TITLE = {Bounds on independent isolation in graphs},
   JOURNAL = {Discrete Appl. Math.},
    VOLUME = {372},
      YEAR = {2025},
     PAGES = {143--149},
       DOI = {10.1016/j.dam.2025.04.016},
       URL = {https://doi.org/10.1016/j.dam.2025.04.016},
}

@article {BoGH,
    AUTHOR = {Boyer, G. and Goddard, W. and Henning, M.A.},
     TITLE = {On total isolation in graphs},
   JOURNAL = {Aequationes Math.},
    VOLUME = {99},
      YEAR = {2025},
    NUMBER = {2},
     PAGES = {623--633},
       DOI = {10.1007/s00010-024-01057-1},
       URL = {https://doi.org/10.1007/s00010-024-01057-1},
}

@article {BKPY,
    AUTHOR = {Bo\v{z}ovi\'{c}, D. and Kelenc, A. and Peterin,
              I. and Yero, I.G.},
     TITLE = {Incidence dimension and 2-packing number in graphs},
   JOURNAL = {RAIRO Oper. Res.},
    VOLUME = {56},
      YEAR = {2022},
    NUMBER = {1},
     PAGES = {199--211},
      ISSN = {2804-7303},
       DOI = {10.1051/ro/2022001},
       URL = {https://doi.org/10.1051/ro/2022001},
}

@article {BDJK,
    AUTHOR = {Bre\v{s}ar, B. and Dravec, T. and Johnston, D.P. and Kuenzel, K. and Rall, D.F. and Tepeh, A.},
    TITLE = {Isolation number: Cartesian and lexicographic products and generalized Sierpiński graphs},
    JOURNAL = {arXiv:2508.16338v1},
    YEAR = {2025},
       DOI = {},
       URL = {https://doi.org/10.48550/arXiv.2508.16338},
}

@article{martinez,
    AUTHOR = {Cabrera Martinez, A. and Peterin, I. and Yero, I.G.},
    TITLE = {Independent transversal total domination versus total domination in trees},
    JOURNAL = {Discuss. Math. Graph Theory},
    VOLUME = {41},
    YEAR = {2021},
    PAGES = {213--224},
}

@article{caro,
  title={Partial domination-the isolation number of a graph},
  author={Caro, Y. and Hansberg, A.},
  journal={Filomat},
  volume={31},
  number={12},
  pages={3925--3944},
  year={2017},
  publisher={JSTOR}
}

@article {ChCZ,
    AUTHOR = {Chen, S. and Cui, Q. and Zhong, L.},
     TITLE = {A characterization of graphs with maximum {$k$}-clique isolation number},
   JOURNAL = {Discrete Math.},
    VOLUME = {348},
      YEAR = {2025},
    NUMBER = {9},
     PAGES = {Paper No. 114531, 13},
       DOI = {10.1016/j.disc.2025.114531},
       URL = {https://doi.org/10.1016/j.disc.2025.114531},
}

@article {CDHH,
    AUTHOR = {Cockayne, E.J. and Dreyer, Jr., P.A. and Hedetniemi, S.M. and Hedetniemi, S.T.},
     TITLE = {Roman domination in graphs},
   JOURNAL = {Discrete Math.},
    VOLUME = {278},
      YEAR = {2004},
    NUMBER = {1-3},
     PAGES = {11--22},
       DOI = {10.1016/j.disc.2003.06.004},
       URL = {https://doi.org/10.1016/j.disc.2003.06.004},
}

@article {Cour,
    AUTHOR = {Courcelle, B.},
     TITLE = {The monadic second-order logic of graphs. I. Recognizable sets of finite graphs},
   JOURNAL = {Inform. Comput.},
    VOLUME = {85},
      YEAR = {1990},
    NUMBER = {1},
     PAGES = {12--75},
       DOI = {doi:10.1016/0890-5401(90)90043-H}
}

@article {CoMR,
    AUTHOR = {Courcelle, B. and Makowsky, J.A. and Rotics, U.},
     TITLE = {Linear time solvable optimization problems on graphs of bounded clique-width},
   JOURNAL = {Theory Comput. Systems},
    VOLUME = {33},
      YEAR = {2000},
    NUMBER = {2},
     PAGES = {125--150},
       DOI = {doi:10.1007/s002249910009}
}

@book{core,
  title={Domination in graphs: Core concepts},
  author={Haynes, T.W. and Hedetniemi, S.T. and Henning, M.A.},
  year={2023},
  publisher={Springer}
}

@article {LMSS,
    AUTHOR = {Lema\'{n}ska, M. and Mora, M. and Souto-Salorio, M.J.},
     TITLE = {Graphs with isolation number equal to one third of the order},
   JOURNAL = {Discrete Math.},
    VOLUME = {347},
      YEAR = {2024},
    NUMBER = {5},
     PAGES = {Paper No. 113903, 10},
       DOI = {10.1016/j.disc.2024.113903},
       URL = {https://doi.org/10.1016/j.disc.2024.113903},
}

@article {LHHF,
    AUTHOR = {Lewis, J. and Hedetniemi, S.T. and Haynes, T.W. and Fricke, G.H.},
     TITLE = {Vertex-edge domination},
   JOURNAL = {Util. Math.},
    VOLUME = {81},
      YEAR = {2010},
     PAGES = {193--213},
}

@article {ZiZy,
    AUTHOR = {Ziemann, R. and \.{Z}yli\'{n}ski, P. },
     TITLE = {Vertex-edge domination in cubic graphs},
   JOURNAL = {Discrete Math.},
    VOLUME = {343},
      YEAR = {2020},
    NUMBER = {11},
     PAGES = {112075, 14},
       DOI = {10.1016/j.disc.2020.112075},
       URL = {https://doi.org/10.1016/j.disc.2020.112075},
}

@article {Zyli,
    AUTHOR = {\.{Z}yli\'{n}ski, P. },
     TITLE = {Vertex-edge domination in graphs},
   JOURNAL = {Aequationes Math.},
    VOLUME = {93},
      YEAR = {2019},
    NUMBER = {4},
     PAGES = {735--742},
       DOI = {10.1007/s00010-018-0609-9},
       URL = {https://doi.org/10.1007/s00010-018-0609-9},
}

@book{Garey,
  author = {Garey, M.R. and Johnson, D.S.},
  title = {Computers and Intractability: A Guide to the Theory of NP-Completeness},
  publisher = {W. H. Freeman},
  year = {1979}
}

@article{lemanska,
  title={Isolation number versus domination number of trees},
  author={Lema{\'n}ska, M. and Souto-Salorio, M.J. and Dapena, A. and Vazquez-Araujo, F.J.},
  journal={Mathematics},
  volume={9},
  number={12},
  pages={1325},
  year={2021},
  publisher={MDPI}
}

\end{document}